\documentclass[11pt]{amsart}
\newtheorem{theorem}{Theorem}[section]
\newtheorem{lemma}[theorem]{Lemma}

\theoremstyle{definition}

\newtheorem{question}[theorem]{Question}
\theoremstyle{remark}
\newtheorem{remark}[theorem]{Remark}

\numberwithin{equation}{section}

\title [Quartic surfaces and Cremona transformations]{Quartic K3 surfaces and Cremona transformations}

\author{Keiji Oguiso}

\address{Keiji Oguiso, Department of Mathematics, Osaka University\\
Toyonaka 560-0043 Osaka, Japan and  Korea Institute for Advanced Study, Hoegiro 87, Seoul, 130-722, Korea} \email{oguiso@math.sci.osaka-u.ac.jp}


\thanks{supported by JSPS Gran-in-Aid (B) No 22340009, JSPS Grant-in-Aid (S), No 22224001, and by KIAS Scholar Program}

\begin{document}

\maketitle

\begin{abstract}
We prove that there is a smooth quartic K3 surface automorphism that is not derived from the Cremona transformation of the ambient three-dimensional projective space. This gives a negative answer to a question of Professor Marat 
Gizatullin. 
\end{abstract}

\section{Introduction}
Throughout this note, we work over the complex number field ${\mathbf C}$. 

In his lecture ``Quartic surfaces and Cremona transformations" in the workshop on Arithmetic and Geometry of K3 surfaces and Calabi-Yau threefolds held at the Fields Institute (August 16-25, 2011), Professor Igor Dolgachev discussed the following question with several beautiful examples supporting it:

\begin{question}\label{dolgachev}
Let $S \subset {\mathbf P}^3$ be a smooth quartic K3 surface. 
Is any biregular automorphism $g$ of $S$ (as abstract variety) derived from a Cremona transformation of the ambient space ${\mathbf P}^3$? More precisely, 
is there a birational automorphism $\tilde{g}$ of ${\mathbf P}^3$ such that 
$\tilde{g}_*(S) =S$ and $g = \tilde{g} \vert S$? 
Here $\tilde{g}_*(S)$ is the proper transform of $S$ and $\tilde{g} \vert S$ 
is the, necessarily biregular, birational automorphism of $S$ then induced by 
$\tilde{g}$. 
\end{question}

Later, Dolgachev pointed out to me that, to his best knowledge, Gizatullin was the first who asked this question. The aim of this short note is to give a negative answer to the question:

\begin{theorem}\label{main}

(1) There exists a smooth quartic K3 surface 
$S \subset {\mathbf P}^3$ of Picard number $2$ such that 
${\rm Pic}\, (S) = {\mathbf Z}h_1 \oplus {\mathbf Z}h_2$ with intersection 
form:
$$((h_i.h_j)) = 
\left(\begin{array}{rr}
4 & 20\\
20 & 4\\
\end{array} \right)\,\, .$$

(2) Let $S$ be as above. Then ${\rm Aut}\,(S)$ has an element $g$ 
such that it is of infinite order and $g^*(h) \not= h$. Here ${\rm Aut}\,(S)$ is the group of biregular automorphisms of $S$ as an abstract variety and $h \in {\rm Pic}\,(S)$ is the hyperplane section class. 

(3) Let $S$ and $g$ be as above. Then there is no element $\tilde{g}$ of ${\rm Bir}\, ({\mathbf P}^3)$ such that $\tilde{g}_{*}(S) = S$ and $g = \tilde{g} \vert S$. Here ${\rm Bir}\, ({\mathbf P}^3)$ is the Cremona group of 
${\mathbf P}^3$, i.e., the group of birational autmorphisms of ${\mathbf P}^3$. \end{theorem}

Our proof is based on a result of Takahashi concerning the 
log Sarkisov program 
(\cite{Ta}), which we quote as Theorem (\ref{logsarkisov}), and standard argument concerning K3 surfaces. 

\begin{remark}\label{dimension}

(1) Let $C \subset {\mathbf P}^2$ be a smooth cubic curve, i.e., a smooth 
curve of genus $1$. It is classical that 
any element of ${\rm Aut}\, (C)$ is derived from a Cremona transformation of 
the ambient space ${\mathbf P}^2$. In fact, this follows from the fact that 
any smooth cubic curve is written in Weierstrass form after a linear change of 
coordinates and the explicit form of the group law 
in terms of the coordinates.

(2) Let $n$ be an integer such that $n \ge 3$ and $Y \subset {\mathbf P}^{n+1}$ be a smooth hypersurface of degree $n+2$. Then $Y$ is an $n$-dimensional Calabi-Yau manifold. It is well-known that ${\rm Bir}\,(Y) = {\rm Aut}\, (Y)$, it is a finite group and any element of ${\rm Aut}\, (Y)$ is derived from a biregular automorphism of the ambient space ${\mathbf P}^{n+1}$. In fact, the statement follows from $K_Y = 0$ in ${\rm Pic}\, (Y)$ (adjunction formula), $H^0(T_Y) = 0$ 
(by $T_Y \simeq \Omega_Y^{n-1}$ together with Hodge symmetry), and ${\rm Pic}\, (Y) = {\mathbf Z} h$, where $h$ is the hyperplane class (Lefschetz hyperplane section theorem). We note that $K_Y = 0$ implies that any birational automorphism of $Y$ is an isomorphism in codimension one, so that for any birational automorphism $g$ of $Y$, we have a well-defined group 
isomorphism $g^*$ on ${\rm Pic}\, (Y)$. Then $g^*h = h$. This implies that 
$g$ is biregular and it is derived from an element of ${\rm Aut}\, ({\mathbf P}^{n+1}) = {\rm PGL}\, ({\mathbf P}^{n+1})$. 

\end{remark}

{\bf Acknowledgement.} I would like to express my thank to Professor Igor Dolgachev for his kind explanation on this question with history, and the organizers of the workshop ``Arithmetic and Geometry of K3 surfaces and Calabi-Yau threefolds" held at the Fields Institute (August 16-25, 2011) for an invitation with full financial support. I would like to express my thanks to the referee for many valuable comments including the simplification of the proof of Lemma (\ref{autom}). 

\section{Proof of Theorem (\ref{main}) (1), (2)}

In this section, we shall prove Theorem (\ref{main})(1)(2) by dividing it into several steps. The last lemma (Lemma (\ref{num})) will be used also in the proof of Theorem (\ref{main}) (3). 

\begin{lemma}\label{morrison} There is a projective K3 surface such that ${\rm Pic}\, (S) = {\mathbf Z}h_1 \oplus {\mathbf Z}h_2$ with 
$$((h_i.h_j)) = 
\left(\begin{array}{rr}
4 & 20\\
20 & 4\\
\end{array} \right)\,\, .$$
\end{lemma}

\begin{proof} Note that the abstract lattice given by the symmetric matrix above is an even lattice of rank $2$ with signature $(1,1)$. Hence the result follows from \cite{Mo}, Corollary (2.9), which is based on the surjectivity of the period map for K3 surfaces (see e.g. \cite{BHPV}, Page 338, Theorem 14.1) and Nikulin's theory (\cite{Ni}) of integral bilinear forms. 
\end{proof}

{\it From now on, $S$ is a K3 surface in Lemma (\ref{morrison}).}

Note that 
the cycle map $c_1 : {\rm Pic}\, (S) \rightarrow {\rm NS}\, (S)$ 
is an isomorphism for a K3 surface. 
So, we identify these two spaces. ${\rm NS}\,(S)_{\mathbf R}$ is 
${\rm NS}\,(S) \otimes_{\mathbf Z} {\mathbf R}$. 
The positive cone $P(S)$ of $S$ is the connected component of the set 
$$\{x \in {\rm NS}\, (S)_{\mathbf R}\, \vert\, (x^2)_S > 0 \}\,\, ,$$ 
containing the ample classes. The ample cone ${\rm Amp}\, (S) \subset 
{\rm NS}\, (S)_{\mathbf R}$ of $S$ is the 
open convex cone generated by the ample classes. 

\begin{lemma}\label{rational curve} ${\rm NS}\,(S)$ represents neither $0$ nor $-2$. In particular, $S$ has no smooth rational curve and no smooth elliptic curve and $(C^2)_S > 0$ for all non-zero effective curves $C$ in $S$. 
In particular, the positive cone of $S$ coincides with the ample cone of $S$.
\end{lemma}

\begin{proof} We have $((xh_1 + yh_2)^2)_S = 4(x^2 + 10xy + y^2)$. Hence there is no $(x, y) \in {\mathbf Z}^2$ such that $((xh_1 + yh_2)^2)_S \in 
\{-2, 0\}$.
\end{proof} 

\begin{lemma}\label{veryample} After replacing $h_1$ by $-h_1$, the line bundle $h_1$ is very ample. In particular, $\Phi_{\vert h_1 \vert} : S \rightarrow {\mathbf P}^3$ is an isomorphism onto a smooth quartic surface.
\end{lemma}

\begin{proof} $h_1$ is non-divisible in ${\rm Pic}\, (S)$ by construction. 
It follows from Lemma (\ref{rational curve}) and $(h_1^2)_S = 4 >0$ 
that one of $\pm h_1$ is ample with no fixed component. By replacing $h_1$ by $-h_1$, we may assume that it is $h_1$. Then, by \cite{SD}, Theorem 6.1, $h_1$ is a very ample line bundle with the last assertion.
\end{proof}

By Lemma (\ref{veryample}), {\it we may and will 
assume that $S \subset {\mathbf P}^3$ 
and denote this inclusion by $\iota$, and a general hyperplane section by $h$. 
That is, $h = H \cap S$ for a general hyperplane $H \subset {\mathbf P}^3$, from now on.} Note that $h = h_1$ in ${\rm Pic}\,(S)$. 

\begin{lemma}\label{autom}
There is an automorphism $g$ of $S$ such that $g$ is of infinite order 
and $g^*(h) \not= h$ in ${\rm Pic}\, (S)$.
\end{lemma}

There are several ways to prove this fact. The following simpler proof was suggested by the referee.

\begin{proof} Let us consider the following orthogonal transformation $\sigma$ 
of ${\rm NS}\, (S))$:
$$\sigma(h_1) = 10h_1 -h_2\,\, ,\,\, \sigma(h_2) = h_1\,\, .$$
It is straightforward to see that $\sigma$ is certainly an element of 
${\rm O}({\rm NS}\, (S))$ and preserves the positive cone of $S$, which is also an ample cone of $S$ by Lemma (\ref{rational curve}). Note also that $\sigma$ is of infinite order, because one of the eigenvalues is $5 + 4\sqrt{6} >1$.

Let $n$ be a positive integer such that $\sigma^n = id$ on the discriminant group $({\rm NS}\, (S))^{*}/{\rm NS}\, (S)$. Such an $n$ exists as $({\rm NS}\, (S))^{*}/{\rm NS}\, (S)$ is a finite set. Let $T(S)$ be the transcendental lattice of $S$, i.e., the orthogonal complement of ${\rm NS}\, (S))$ in $H^2(S, {\mathbf Z})$. Then, by \cite{Ni}, Proposition 1.6.1, the isometry $(\sigma^n, id_{T(S)})$ of ${\rm O}({\rm NS}\, (S)) \times {\rm O}(T(S))$ extends to an isometry $\tau$ of $H^2(S, {\mathbf Z})$. Since $\tau$ also 
preserves the Hodge decomposition and the ample cone, there is then an automorphism $g$ of $S$ such that $g^* = \tau$ by the global Torelli theorem for K3 surfaces (see eg. \cite{BHPV}, Chapter VIII). This $g$ satisfies the requirement.
\end{proof} 

Let $g$ be as in Lemma (\ref{autom}). Then the pair $(S \subset {\mathbf P}^3, g)$ satisfies all the requirements of Theorem (\ref{main})(1), (2). 

\begin{lemma}\label{num}
Let $(S \subset {\mathbf P}^3, g)$ be as in Theorem (\ref{main})(1),(2). 
Let $C \subset S$ be a non-zero effective curve of degree $< 16$, i.e., 
$$(C \cdot h)_S = (C \cdot H)_{{\mathbf P}^3} < 16\,\, .$$
Then $C = S \cap T$ for some hypersurface $T$ in ${\mathbf P}^3$. 
\end{lemma}

\begin{proof} Recall that $h = h_1$ in ${\rm Pic}\, (S)$. 
There are $m, n \in {\mathbf Z}$ such that $C = mh_1 +nh_2$ 
in ${\rm Pic}\, (S)$. Then 
$$(C \cdot h)_S = 4(m + 5n) > 0\,\, ,\,\, (C^2)_S = 4(n^2 + 10mn + m^2) > 0\,\, .$$
Here the last inequality follows from Lemma (\ref{rational curve}). 
Thus, if $(C \cdot h)_S < 16$, then $m+5n$ is either $1$, $2$ or $3$ by $m, n \in {\mathbf Z}$. Hence we have either one of 
$$m = 1-5n\,\, ,\,, m = 2-5n\,\, ,\,\, m = 3-5n\, .$$
Substituting into $n^2 + 10mn + m^2 > 0$, we obtain one of either 
$$1 -24n^2 > 0\,\, ,\,\, 4-24n^2 >0\,\, ,\,\, 9-24n^2 >0\,\, .$$
Since $n \in {\mathbf Z}$, it follows that $n =0$ in each case. 
Therefore, in ${\rm Pic}\,(S)$, we have $C = mh$ for some 
$m \in {\mathbf Z}$. 
Since $H^1({\mathbf P}^3, {\mathcal O}_{{\mathbf P}^3}(\ell)) = 0$ for all $\ell \in {\mathbf Z}$, the natural restriction map 
$$\iota^* : H^0({\mathbf P}^3, {\mathcal O}_{{\mathbf P}^3}(m)) \rightarrow H^0(S, {\mathcal O}_S(m))$$
is surjective for all $m \in {\mathbf Z}$. This implies the result. 
\end{proof} 
\section{Proof of Theorem (\ref{main}) (3)}

In his paper \cite{Ta}, Theorem 2.3 and Remark 2.4, N. Takahashi proved the 
following remarkable theorem as a nice application of the log Sarkisov program (For terminologies, we refer to \cite{KMM}):

\begin{theorem}\label{logsarkisov}
Let $X$ be a Fano manifold of dimension $\ge 3$ with Picard number $1$, 
$S \in \vert -K_X \vert$ be a smooth hypersurface. Let $\Phi : X \cdots \to X'$ be a birational map to a ${\mathbf Q}$-factorial termial variety $X'$ with Picard number $1$, which is not an isomorphism, and $S' := \Phi_*S$. Then:

(1) If ${\rm Pic}\, (X) \rightarrow {\rm Pic}\, (S)$ is surjectve, then $K_{X'} + S'$ is ample. 

(2) Let $X = {\mathbf P}^3$ and $H$ be a hyperplane of ${\mathbf P}^3$. Note 
that then $S$ is a smooth quartic K3 surface. Assume that any irreducible reduced curve $C \subset S$ such that $(C \cdot H)_{{\mathbf P}^3} < 16$ is of the form $C = S \cap T$ for some hypersurface $T \subset {\mathbf P}^3$. Then $K_{X'} + S'$ is 
ample. 
\end{theorem}

Applying Theorem (\ref{logsarkisov})(2), we shall complete the proof of 
Theorem (\ref{main})(3) in the following slightly generalized form:

\begin{theorem}\label{gen}
Let $S \subset {\mathbf P}^3$ be a smooth quartic K3 surface. Then:

(1) Any automorphism $g$ of $S$ of infinite order is not the restriction of a 
biregular automorphism of the ambient space ${\mathbf P}^3$, i.e., the restriction of an element of ${\rm PGL}(\mathbf P^3)$. 

(2) Assume further that $S$ contains 
no curves of degree $< 16$ which are not cut out by a hypersurface. 
Then, any automorphism $g$ of $S$ of infinite order is not the restriction 
of a 
Cremona transformation of the ambient space ${\mathbf P}^3$.
\end{theorem} 

Recalling Lemma (\ref{num}), we see that the pair 
$(S \subset {\mathbf P}^3, g)$ in 
Theorem (\ref{main})(1)(2) satisfies all the requirements of 
Theorem (\ref{gen})(2). So, Theorem (\ref{main})(3) follows from Theorem (\ref{gen}) (2). We prove Theorem (\ref{gen}). 

\begin{proof} Let us first show (1). Consider the group $G := \{g \in {\rm PGL}(\mathbf P^3)\, \vert\, g(S) = S\}$. Let $H$ be the connected component of 
${\rm Hilb}\,({\mathbf P}^3)$ containing $S$. Then $G$ 
is the stabilizer group of the point $[S] \in H$ 
under the natural action of ${\rm PGL}(\mathbf P^3)$ on 
$H$. In particular, $G$ is a Zariski closed subset of the affine variety ${\rm PGL}(\mathbf P^3)$. In particular, $G$ has only finitely many irreducible components. Note that the natural map $G \rightarrow {\rm Aut}\, (S)$ is injective and $H^0(S, T_S) = 0$. Thus $\dim\, G = 0$. Hence 
$G$ is a finite set. 

Let $g \in {\rm Aut}\, (S)$. If there is an element $\tilde{g} \in {\rm PGL}\,({\mathbf P}^3)$ such that $g = \tilde{g} \vert S$, then 
$g \in G$, and therefore $g$ is of finite order. This proves (1). 

Let us show (2). 
We argue by contradiction, 
i.e.,  {\it assuming to the contrary that there would be  
a birational map $\tilde{g} : {\mathbf P}^3 \cdots \to {\mathbf P}^3$ 
such that $\tilde{g}_*(S) = S$ and 
that $g = \tilde{g} \vert S$, we shall derive a contradiction.} 

We shall divide it into two cases:

(i) $\tilde{g}$ is an isomorphism, (ii) $\tilde{g}$ is not an isomorphism.

Case (i). By (1), $g$ would be of finite order, a contradiction. 

Case (ii). By the case assumption, our $S$ would satisfy 
all the conditions of Theorem (\ref{logsarkisov})(2). Recall also that $\tilde{g}_*S = S$. However, then, 
by Theorem (\ref{logsarkisov})(2), $K_{{\mathbf P}^3} + S$ would be ample, a contradiction to $K_{{\mathbf P}^3} + S = 0$ in ${\rm Pic}\, ({\mathbf P}^3)$. 

This completes the proof.
\end{proof}

\end{document}